\documentclass{amsart}

\newtheorem{thm}{Theorem}

\newtheorem{lem}{Lemma}

\newcommand{\A}{{\mathcal A}}

\newcommand{\U}{{\mathcal U}}

\newcommand{\IC}{{\mathbb C}}
\newcommand{\ID}{{\mathbb D}}

\newcommand{\D}{{\mathbb D}}

\def\be{\begin{equation}}
\def\ee{\end{equation}}

\begin{document}

\title[First-order Schwartian derivative and some classes]{First-order Schwartian derivative and some classes of univalent functions }

\author[M. Obradovi\'{c}]{Milutin Obradovi\'{c}}
\address{Department of Mathematics,
Faculty of Civil Engineering, University of Belgrade,
Bulevar Kralja Aleksandra 73, 11000, Belgrade, Serbia}
\email{obrad@grf.bg.ac.rs}

\author[N. Tuneski]{Nikola Tuneski}
\address{Department of Mathematics and Informatics, Faculty of Mechanical Engineering, Ss. Cyril and Methodius,
University in Skopje, Karpo\v{s} II b.b., 1000 Skopje, Republic of North Macedonia.}
\email{nikola.tuneski@mf.edu.mk}



\maketitle

\begin{abstract}
In this paper we give estimate of $|S'_{f}(0)|$, where
$S_{f}(z)= \left(\frac{f''(z)}{f'(z)}\right)'-\frac{1}{2}\left(\frac{f''(z)}{f'(z)}\right)^{2}$ is the Schwartian derivative, and $f$ belongs to different classes of functions univalent in the open unit disc $\ID$.
\end{abstract}

\section{Introduction}

Let ${\mathcal A}$ denote the family of all functions analytic in the open unit disk $\ID := \{ z\in \IC:\, |z| < 1 \}$,  satisfying the normalization $f(0)=0= f'(0)-1$, i.e.,
\begin{equation}\label{eq-1}
f(z)=z+a_2z^2+a_3z^3+\cdots.
\end{equation}
Let $\mathcal{S}$ be the subclass of $\mathcal{A}$ consisting of functions that are univalent in $\ID$.

\medskip

Next we will define several classes of univalent functions that will be studied later in the paper.

\medskip

For $0\leq \alpha<1$, let $\mathcal{S}^{\star}(\alpha)$ and $\mathcal{C}(\alpha)$ denote the subclasses of
${\mathcal A}$ which are starlike of order $\alpha$ and convex of order $\alpha$ in $\ID$, respectively, i.e., with analytical characterisation:
\[ \mathcal{S}^{\star}(\alpha) = \left\{ f\in\mathcal{A}: \operatorname{Re}\left[\frac{zf'(z)}{f(z)}\right] >\alpha,\, z\in\ID \right\} \]
and
\[ \mathcal{C}(\alpha) = \left\{ f\in\mathcal{A}: \operatorname{Re}\left[1+\frac{zf''(z)}{f'(z)}\right] >\alpha,\, z\in\ID \right\}.\]

\medskip

Further, let  $\mathcal{S}\mathcal{S}^{\star}_{\beta}$ and  $\mathcal{S}\mathcal{C}_{\beta}$, $0<\beta\le1$,
be the subclasses of $\mathcal{A}$ consisting of functions $f$ satisfying
\[
\left|\arg \left(\frac{zf'(z)}{f(z)}\right)\right|<\frac{\pi\beta}{2}\quad (z\in \ID)
\]
and
\[
\left|\arg \left(1+\frac{zf''(z)}{f'(z)}\right)\right|<\frac{\pi\beta}{2}\quad (z\in \ID),
\]
respectively. The functions from $\mathcal{S}\mathcal{S}^{\star}_{\beta}$ are called strongly starlike of order $\beta$, and from $\mathcal{S}\mathcal{C}_{\beta}$
strongly convex of order $\beta$.

\medskip

Similarly, $\mathcal{G}(\alpha)$, $0<\alpha\leq 1$, is the class of functions $f\in \mathcal{A}$ for which
\[
{\rm Re}\left[1+\frac{zf''(z)}{f'(z)}\right] <1+\frac{1}{2}\alpha \quad (z\in\ID).
\]
Ozaki in \cite{ozaki-1941} introduced the class $\mathcal{G}(1)$ and proved that functions in $\mathcal{G}(1)$ are univalent in the unit disk. Later, Umezawa in \cite{umezawa}, Sakaguchi in \cite{saka} and R. Singh and S. Singh in \cite{singh} showed, respectively,  that functions in $\mathcal{G}(1)$ are convex in one direction, close-to-convex and starlike. This, general class is extensively studied  in  \cite{MO-1995} and \cite{MO-2013}.

\medskip

Further,
\[
\U(\lambda) = \left\{ f\in\A: \left|\left(\frac{z}{f(z)}\right)^2f'(z)-1\right| < \lambda, z\in\ID  \right\},
\]
where $0<\lambda\le1$, is subclass of the class of univalent functions, and its special case when $\lambda=1$ is first studied in \cite{Japonica_1996}. More details on them can be found in \cite{obpon-1,obpon-2,TTV}).

\medskip

Finally, for $0\leq \alpha<1$,
\[
\mathcal{R}(\alpha)=\left\{f\in\mathcal{A}: {\rm Re}f'(z)>\alpha,\, z\in\ID \right\}
\]
is subclass of the class of function of bounded turning $\mathcal{R}(0)$.

\medskip

For a locally univalent function $f$ in $\ID$, the Schwartian derivative $S_{f}$ is defined as
\[
S_{f}(z)= \left(\frac{f''(z)}{f'(z)}\right)'-\frac{1}{2}\left(\frac{f''(z)}{f'(z)}\right)^{2}.
\]
In \cite{QW} it is shown that for the function $f$ given by \eqref{eq-1},
\begin{equation}\label{eq-2}
S'_{f}(0)=24(a_{4}-2a_{2}a_{3}+a_{2}^{3}).
\end{equation}

\medskip

Using some properties of Caratheodory class $\mathcal{P}$ consisting of functions $p(z)= 1+p_{1}z+p_{2}z^{2}+\cdots$ such that ${\rm Re}\{p(z)\}>0$ $z\in \ID$, in \cite{QW} the authors gave the estimate of $|S'_{f}(0)|$ for the classes $\mathcal{S}^{\star}(\alpha)$, $\mathcal{C}(\alpha)$, $\mathcal{S}\mathcal{S}^{\star}_{\beta}$, and  $\mathcal{S}\mathcal{C}_{\beta}$.

\medskip

In this paper, using different approach, we will give the estimates of $|S'_{f}(0)|$ for the classes of univalent functions defined above, as well as, for the general class $\mathcal{S}$.

\medskip

\section{The estimate of $|S'_{f}(0)|$ for some classes of univalent functions}

\medskip

In the study of $|S'_{f}(0)|$ we will need the following Lemma from \cite{126}, given using the notation from that paper.

\medskip

\begin{lem}\label{lem-1} If $\omega(z)=c_{1}z+c_2z^2+c_3z^3+\cdots $ is analytic in $\D$ satisfying the condition
$|\omega(z)|<1$, $z\in\D $, and if
$$ \Psi(\omega)=|c_{3}+\mu c_{1}c_{2}+\nu c_{1}^{3}|,$$
then the following sharp estimate $\Psi(\omega)\leq \Phi(\mu, \nu)$ holds, where
\[\Phi(\mu,\nu)=
\begin{cases}
1, & (\mu,\nu)\in D_{1}\cup D_{2}\cup\{(2,1)\},\\
|\nu|, & (\mu,\nu)\in D_{6},\\
\end{cases}\]
where
$$D_{1}=\left\{(\mu, \nu): |\mu|\leq\frac{1}{2},\,-1\leq\nu\leq 1\right\} ,$$
$$D_{2}=\left\{(\mu, \nu): \frac{1}{2}\leq |\mu|\leq 2,\,\frac{4}{27}(|\mu|+1)^{3}-(|\mu|+1)\}\leq\nu\leq 1 \right\},$$
and
$$D_{6}=\left\{(\mu, \nu): 2\leq|\mu|\leq4,\,\nu\geq\frac{1}{12}(\mu^{2}+8)\right\}.$$
\end{lem}

\medskip

\begin{thm}\label{th1}
Let $f\in\mathcal{R}(\alpha)$, $0\leq \alpha<1$, then $|S'_{f}(0)|\leq 12(1-\alpha).$
This result is sharp
\end{thm}
\begin{proof}
From the definition of the class $\mathcal{R}(\alpha)$, we have
\[
f'(z)=\alpha +(1-\alpha)\frac{1+\omega(z)}{1-\omega(z)} \quad \left( =2\alpha -1+2(1-\alpha)\frac{1}{1-\omega(z)}\right),
\]
where $\omega$ is analytic in $\ID$ with $\omega(0)=0$ and $|\omega(z)|<1$, $z\in\ID$. From here
\[
f'(z)= 1+2(1-\alpha)\left[\omega(z)+\omega^{2}(z)+\cdots\right].
\]
If we put $\omega(z)=c_{1}z+c_{2}z^{2}+\cdots$, and compare the coefficients on $z$, $z^{2}$, and $z^{3}$, after some calculations, we obtain
\be\label{eq-5}
\begin{split}
a_{2}&=(1-\alpha)c_{1}, \\
a_{3}&=\frac{2}{3}(1-\alpha)\left(c_{2}+c_{1}^{2}\right),\\
a_{4}&=\frac{1}{2}(1-\alpha)\left(c_{3}+2 c_{1}c_{2}+ c_{1}^{3}\right).
\end{split}
\ee
Now, using \eqref{eq-2} and \eqref{eq-5} we get
\[
\begin{split}
|S'_{f}(0)|&=24|a_{4}-2a_{2}a_{3}+a_{2}^{3}|\\
&=12(1-\alpha)\left|c_{3}+\frac{2}{3}(4\alpha-1)c_{1}c_{2}+\frac{6\alpha^{2}-4\alpha+1}{3}c_{1}^{3}\right|\\
&=12(1-\alpha)\cdot \Psi(\omega),
\end{split}
\]
where we used notation from Lemma \ref{lem-1}, with $\mu = \frac{2}{3}(4\alpha-1)$ and $\nu = \frac{6\alpha^{2}-4\alpha+1}{3}$.
Using the lemma, to prove the estimate, it is enough to show that for $0\le\alpha<1$, $(\mu,\nu)$ is in $D_1$ or in $D_2$, which will mean that $\Psi(\omega)\le \Phi(\mu,\nu)=1$.

\medskip

Indeed, if $1/16\le\alpha<7/16$, or equivalently $|\mu|\le1/2$, then $0\le\nu\le1$, i.e., $(\mu,\nu)$ is in $D_1$.

\medskip

Otherwise, if $0\le\alpha\le1/16$ or $7/16\le\alpha\le1$, we have $1/2\le|\mu|\le2$ and easily $\nu\le1$.
In the first case, if $0\le\alpha\le1/16$, then $\mu<0$, while in the second case, if $7/16\le\alpha\le1$, then $\mu>0$. In both cases it can be easily verified that
\[ \frac{4}{27}(|\mu|+1)^{3}-(|\mu|+1)\}\leq\nu,\]
leading to the conclusion  $(\mu,\nu)\in D_2$.

\medskip

The result is sharp for the function for which $f'(z)=\alpha +(1-\alpha)\frac{1+z^{3}}{1-z^{3}}.$
\end{proof}

\medskip

\begin{thm}\label{th2}
Let $f\in\mathcal{C}(\alpha)$ $-\frac{1}{2}\leq \alpha<1$. Then $|S'_{f}(0)|\leq 4(1-\alpha)$, and the bound is sharp.
\end{thm}

\begin{proof} From the definition of the class $\mathcal{C}(\alpha)$, similarly as in the proof of Theorem \ref{th1}, we have
\be\label{eq-6}
\left( z f'(z)\right)'=\left[1+2(1-\alpha)(\omega(z)+\omega^{2}(z)+\cdots)\right]\cdot f'(z),
\ee
where $\omega$ is analytic in $\ID$ with $\omega(0)=0$ and $|\omega(z)|<1$, $z\in\ID$.
If we put $\omega(z)=c_{1}z+c_{2}z^{2}+\cdots$, and compare the coefficients of $z$, $z^{2}$, and $z^{3}$, from the relation \eqref{eq-6} after some simple calculations we obtain,
\be\label{eq-7}
\begin{split}
a_{2}&=(1-\alpha)c_{1}, \\
a_{3}&=\frac{1}{3}(1-\alpha)\left(c_{2}+(3-2\alpha)c_{1}^{2}\right),\\
a_{4}&=\frac{1}{6}(1-\alpha)\left(c_{3}+(5-3\alpha) c_{1}c_{2}+(2\alpha^{2}-7\alpha +6) c_{1}^{3}\right).
\end{split}
\ee
Next, using \eqref{eq-2} and \eqref{eq-7} we can have
\[
\begin{split}
|S'_{f}(0)|&=24|a_{4}-2a_{2}a_{3}+a_{2}^{3}|\\
&=4(1-\alpha)\left|c_{3}+(\alpha+1)c_{1}c_{2}+\alpha c_{1}^{3}\right|\\
&\leq 4(1-\alpha),
\end{split}
\]
where we used Lemma \ref{lem-1}, case $D_{2}$. Namely, for $-\frac{1}{2}\leq \alpha<1$ we have $ \frac{1}{2}\leq \mu=\alpha +1\leq 2 $, and we have
$\frac{4}{27}(|\mu|+1)^{3}-(|\mu|+1)\leq\nu\leq 1$ being equivalent to
$\frac{4}{27}(\alpha+2)^{3}-(\alpha +2)\leq\alpha\leq 1$, i.e., to the inequality
$(\alpha-1)(2\alpha^{2}+14\alpha+11)\leq0$ which can be easily verified to be true.

\medskip

Given result is sharp for the function satisfying
$1+\frac{zf''(z)}{f'(z)}= \alpha +(1-\alpha)\frac{1+z^{3}}{1-z^{3}}$, or equivalently,
$f(z)=\int_{0}^{z}(1-t^{3})^{-\frac{2}{3}(1-\alpha)}dt .$

\medskip

We note that in \cite{QW} the authors gave the same result, but only for the case $0\leq \alpha<1$
\end{proof}

\begin{thm}\label{th3}
Let $f\in\mathcal{G}(\alpha),\,0<\alpha\leq1$, then $|S'_{f}(0)|\leq \alpha^{2}+2\alpha$,
and this result is sharp.
\end{thm}

\begin{proof}
From the definition of the class $\mathcal{G}(\alpha)$ we have
\[
1+\frac{zf''(z)}{f'(z)}=1+\frac{1}{2}\alpha -\frac{\alpha}{2}\frac{1+\omega(z)}{1-\omega(z)}
\quad\left(=1+\alpha-\alpha\frac{1}{1-\omega(z)} \right),\]
where $\omega$ is analytic in $\ID$ with $\omega(0)=0$ and $|\omega(z)|<1$, $z\in\ID$.
The last relation can be rewritten as
\be\label{eq-8}
[z f'(z)]'=\left\{1-\alpha[\omega(z)+\omega^{2}(z)+\cdots]\right\}\cdot f'(z).
\ee
Similarly as in the proofs of previous theorems, putting $\omega(z)=c_{1}z+c_{2}z^{2}+\cdots$ in \eqref{eq-8} and comparing  the coefficients on $z$, $z^{2}$, and $z^{3}$, leads to
\be\label{eq-9}
\begin{split}
\displaystyle\smallskip
a_{2}&=-\frac{\alpha}{2}c_{1}, \\
a_{3}&=-\frac{\alpha}{6}\left[c_{2}+(1-\alpha) c_{1}^{2}\right],\\
a_{4}&=-\frac{\alpha}{24}\left[2c_{3}+(4-3\alpha)c_{1}c_{2}+(\alpha^{2}-3\alpha+2)c_{1}^{3}\right].
\end{split}
\ee
From \eqref{eq-2} and \eqref{eq-9}, after some calculations, we get
\[
\begin{split}
|S'_{f}(0)|&=24|a_{4}-2a_{2}a_{3}+a_{2}^{3}|\\
&=2\alpha\left|c_{3}+\left(2+\frac{1}{2}\alpha\right)c_{1}c_{2}+\left(1+\frac{1}{2}\alpha\right) c_{1}^{3}\right|\\
&\leq 2\alpha \left(1+\frac{1}{2}\alpha\right)\\
&=\alpha^{2}+2\alpha.
\end{split}
\]
In last step we applied Lemma \ref{lem-1}, case $D_{6}$, since $2<\mu=2+\frac{1}{2}\alpha\leq \frac{5}{2}$ and
$1<\nu= 1+\frac{1}{2}\alpha\leq \frac{3}{2}$. The result of this theorem is sharp as the function defined by
\[
1+\frac{zf''(z)}{f'(z)}=1+\frac{1}{2}\alpha -\frac{\alpha}{2}\cdot \frac{1+z}{1-z},
\]
shows, i.e., for $f(z)=\frac{1-(1-z)^{\alpha +1}}{\alpha +1}.$
\end{proof}

\medskip

For the proof of the next theorem we will need the following result proven in \cite{OP-2019} as a part of the proof of its Theorem 1.

\medskip

 \begin{lem}
 For each function $f$ in $\mathcal{U}(\lambda),$ $0< \lambda \leq 1,$ there exists function $\omega ,$ analytic in $\mathbb{D},$ such that $|\omega (z)|\leq |z| < 1,$ and $|\omega '(z)|\leq 1,$
 for all $z \in \mathbb{D},$ with
 \begin{equation}\label{eq-10}
 \frac{z}{f(z)} =1-a_2z-\lambda z\omega (z).
\end{equation}
 Additionally, for $\omega(z)=c_1z+c_2z^2+\cdots,$
 \begin{equation}\label{eq-11}
 |c_1|\leq 1\quad \mbox{and}\quad |c_2|\leq \frac{1}{2}(1-|c_1|^2).
 \end{equation}
\end{lem}

\medskip

\begin{thm}\label{th4}
If $f$ belongs to $\mathcal{U}(\lambda),$ $0< \lambda \leq 1,$ then $|S'_{f}(0)|\leq 12\lambda$
and the result is sharp.
\end{thm}

\begin{proof}
Using \eqref{eq-10} we have
\[
z =[1-a_2z-\lambda z\omega(z)]\cdot f(z),
\]
and after equating the coefficients,
\[
\begin{split}
a_3 =& \lambda c_1+a^2_2,\\
a_4=&\lambda c_2+2\lambda a_2c_1+a_2^3.
\end{split}
\]
Finally, using previous relations and \eqref{eq-2}, we have
\[
\begin{split}
|S'_{f}(0)|&=24|a_{4}-2a_{2}a_{3}+a_{2}^{3}|=24|\lambda c_{2}|\\
&\leq 24\lambda\cdot \frac{1}{2}(1-|c_1|^2) \leq12\lambda,
\end{split}
\]
where we used \eqref{eq-11}. The result is sharp with extremal function
$f(z)=\frac{z}{1-\frac{\lambda}{2}z^{3}}.$
\end{proof}

\medskip

\section{Estimate of $|S'_{f}(0)|$ for the general class $\mathcal{S}$ of univalent functions}

\medskip

For our investigation in this section we  will use the Grunsky coefficients and its properties. Here are basic definitions and results  on those coefficients based on the book of N.A. Lebedev (\cite{Lebedev}).

\medskip

Let $f \in \mathcal{S}$ and let
\[
\log\frac{f(t)-f(z)}{t-z}=\sum_{p,q=0}^{\infty}\omega_{p,q}t^{p}z^{q},
\]
where $\omega_{p,q}$ are called Grunsky's coefficients with property $\omega_{p,q}=\omega_{q,p}$.
For those coefficients we have the next Grunsky's inequalitiy (\cite{duren,Lebedev}):
\be\label{eq-12}
\sum_{q=1}^{\infty}q \left|\sum_{p=1}^{\infty}\omega_{p,q}x_{p}\right|^{2}\leq \sum_{p=1}^{\infty}\frac{|x_{p}|^{2}}{p},
\ee
where $x_{p}$ are arbitrary complex numbers such that $0< \sum_{p=1}^{\infty}\frac{|x_{p}|^{2}}{p}< +\infty$. If $\overline{\lim}_{p\to\infty} \sqrt[p]{|x_p|}<1$, then in \eqref{eq-12} we have equality if, and only if, the area of $\widehat{\mathbb{C}} \setminus f^{-1}(\D)$ is zero, where $f^{-1}(z)= \frac{1}{f(z)}$.

\medskip

Further, it is well-known that if $f$ given by \eqref{eq-1}
belongs to $\mathcal{S}$, then also
\be\label{eq-13}
f_{2}(z)=\sqrt{f(z^{2})}=z +c_{3}z^3+c_{5}z^{5}+\cdots
\ee
belongs to the class $\mathcal{S}$. Then, for the function $f_{2}$ we have the appropriate Grunsky's
coefficients of the form $\omega_{2p-1,2q-1}$, and the inequality \eqref{eq-12} has the form:
\be\label{eq-14}
\sum_{q=1}^{\infty}(2q-1) \left|\sum_{p=1}^{\infty}\omega_{2p-1,2q-1}x_{2p-1}\right|^{2}\leq \sum_{p=1}^{\infty}\frac{|x_{2p-1}|^{2}}{2p-1}.
\ee
Here, and further in the paper we omit the upper index (2) in  $\omega_{2p-1,2q-1}^{(2)}$ if compared with Lebedev's notation.

\medskip

From the inequality \eqref{eq-14}, when $x_{2p-1}=0$ and $p=3,4,\ldots$, we have
\[
|\omega_{11} x_1 +\omega_{31} x_3 |^2 +3|\omega_{13} x_1 +\omega_{33} x_3 |^2 + 5|\omega_{15} x_1 +\omega_{35} x_3 |^2 \le |x_1|^2+\frac{|x_3|^2}{3}.
\]
 From here, for $x_1=1$ and $x_3=0$, we obtain
\begin{equation}\label{eq-16}
  |\omega_{11}|^2 + 3 |\omega_{13}|^2 + 5|\omega_{15}|^2 \leq1.
\end{equation}

\medskip

As it has been shown in \cite[p. 57]{Lebedev}, if $f$ is given by \eqref{eq-1}, then the coefficients $a_{2}$, $ a_{3}$, and $ a_{4}$  are expressed by Grunsky's coefficients  $\omega_{2p-1,2q-1}$ of the function $f_{2}$ given by
\eqref{eq-13} in the following way:
\begin{equation}\label{eq-17}
\begin{split}
a_{2}&=2\omega _{11},\\
a_{3}&=2\omega_{13}+3\omega_{11}^{2}, \\
a_{4}&=2\omega_{33}+8\omega_{11}\omega_{13}+\frac{10}{3}\omega_{11}^{3},\\
0&=3\omega_{15}-3\omega_{11}\omega_{13}+\omega_{11}^{3}-3\omega_{33}.
\end{split}
\end{equation}

\medskip

Now, we can formulate and prove the main result of this section.

\begin{thm}\label{th5}
If $f$ belongs to $\mathcal{S}$, then $|S'_{f}(0)|\leq \frac{64}{5}\sqrt{3}=22.17025\ldots.$
\end{thm}

\begin{proof}
Using \eqref{eq-2} and \eqref{eq-17}, we have
\[
\begin{split}
|S'_{f}(0)|&=24|a_{4}-2a_{2}a_{3}+a_{2}^{3}|\\
&=48\left|\omega_{33}-\frac{1}{3}\omega_{11}^{3}\right|,
\end{split}
\]
or, if we use the fourth relation from \eqref{eq-17} and the relation \eqref{eq-16},
\[
\begin{split}
|S'_{f}(0)|&=48|\omega_{15}-\omega_{11}\omega_{13}|\\
&\leq48(|\omega_{15}|+|\omega_{11}||\omega_{13}|)\\
&\leq 48\left(\frac{1}{\sqrt{5}}\sqrt{1-|\omega_{11}|^2-3|\omega_{13}|^2}+|\omega_{11}||\omega_{13}|\right)\\
&=:48\cdot g(|\omega_{11}|,|\omega_{13}|),
\end{split}
\]
where
$g(x,y)=\frac{1}{\sqrt{5}}\sqrt{1-x^2-3y^2}+x y$ and
\[ (x,y)\in \Omega = \left\{(x,y) : 0\leq x \leq 1,\, 0\leq y\leq \frac{1}{\sqrt3}\sqrt{1-x^2} \right\}. \]

\medskip

Now, let find maximal value of the function $g$ on $\Omega$.

\medskip

To derive the critical points of $g$, we need to solve in the interior of $\Omega$ the following system
\[
\left\{
\begin{split}
\frac{\partial g(x,y)}{\partial x} &= y-\frac{1}{\sqrt{5} }\frac{x}{\sqrt{1-x^2-3 y^2}} = 0,\\
\frac{\partial g(x,y)}{\partial y} &= x-\frac{1}{\sqrt{5} }\frac{3 y}{\sqrt{1-x^2-3 y^2}} = 0.
\end{split}
\right.
\]
Now, from
\[3y\frac{\partial f_2(x,y)}{\partial x} - x \frac{\partial f_2(x,y)}{\partial y} =
3 y^2-x^2 = 0 , \]
we receive
\[y^2=\frac{x^2}{3}.\]
Substituting it in the second equation of the system gives
\[ \frac{x}{\sqrt{3}}-\frac{1}{\sqrt{5} }\frac{x}{\sqrt{1-2 x^2}} = 0,\]
which brings the unique solution of the system in the interior of $\Omega$, $x_1=\frac{1}{\sqrt5}$ and $y_1=\frac{1}{\sqrt{15}}$ such that $g(x_1,y_1)=\frac{64}{5}\sqrt{3}=22.17025\ldots$.

\medskip

On the edges of $\Omega$ we have:
\begin{itemize}
  \item[-] $g(0,y)=\frac{1}{\sqrt5}\sqrt{1-3 y^2}\le \frac{1}{\sqrt5} = 0.4472\ldots;$
  \item[-] $g(1,y) = g(1,0) = 0$;
  \item[-] $g(x,0) = \frac{1}{\sqrt5}\sqrt{1-x^2}\le \frac{1}{\sqrt5} = 0.4472\ldots;$
  \item[-] $h(x):=g\left(x, \sqrt{1-x^2}/\sqrt3\right) = \frac{1}{\sqrt3}x\sqrt{1-x^2}\le h(1/\sqrt2) = \sqrt3/6 = 0.288675\ldots$.
\end{itemize}
Combining the above analysis, we receive that the function $g$, on the domain $\Omega$ achieves the greatest value $\frac{64}{5}\sqrt{3}$ obtained for $x_1=\frac{1}{\sqrt5}$ and $y_1=\frac{1}{\sqrt{15}}$, i.e.,
\[|S'_{f}(0)|\leq \frac{64}{5}\sqrt{3}=22.17025\ldots.\]
\end{proof}

\medskip

\end{document}